\documentclass[10pt]{article}
\usepackage[english]{babel}
\usepackage{amsmath,amsthm}
\usepackage{amsfonts}
\usepackage[none]{hyphenat}
\usepackage{xcolor}
\usepackage[font={small}]{caption}
\usepackage[explicit]{titlesec}
\usepackage{anyfontsize}
\titleformat{\section}{\Large}{\textbf{\thesection .}}{1em}{\textbf{#1}}

\addto{\captionsenglish}{}
\newtheorem{thm}{Theorem}[section]
\newtheorem{cor}[thm]{Corollary}

\newtheorem{prop}[thm]{Proposition}
\theoremstyle{definition}

\theoremstyle{remark}

\numberwithin{equation}{section}

\renewenvironment{proof}{{\vspace{-1em}\quad \textit{Proof.}}}{\hfill$\square$}
\begin{document}
\renewcommand{\thefootnote}{\fnsymbol{footnote}}
\title{Generalized $K$-Shift Forbidden Substrings\\[0.1em] in Permutations}%
\author{\vspace{-3,5\baselineskip}Enrique Navarrete$^{\ast}$}\footnotetext[1]{Grupo ANFI, Universidad de Antioquia.}%
\date{}%
\renewcommand*{\thefootnote}{\arabic{footnote}}
\makeatletter
\def\maketitle{%
\bgroup
\par\centering{\LARGE\@title}\\[3em]%
\ \par
{\@author}\\[4em]%
\egroup
}

\maketitle
\begin{abstract}
  In this note we continue the analysis started in [2] and generalize propositions regarding permutations that avoid substrings\linebreak $12, 23,\ldots , (n-1)n$, (and others) to permutations that for fixed $k$,\linebreak $k < n$, avoid substrings $j(j+k)$, $1 \leq j \leq n-k$, as well as substrings $j(j+k)\ (\text{mod }n)$, $1 \leq j \leq n$, (\textit{ie}. $k$-shifts in general, as defined in Section 2).\\[1em]
\end{abstract}
\section{Introduction and Previous Results}

In this section we summarize some results obtained in [2] and we recall the following definitions:\footnote[2]{Note: In [2], the term \textquotedblleft linear arrangement\textquotedblright\ was used instead of \textquotedblleft permutation\textquotedblright , and \textquotedblleft pattern\textquotedblright\ instead of \textquotedblleft substring\textquotedblright. Here we use the more conventional terminology. Permutations are meant to be in one-line notation.}
\begin{description}
  \item[$d_n$]  :=  the number of permutations on $[n]$ that avoid substrings\linebreak $12, 23,\ldots , (n-1)n$.
  \item[$D_n$] :=  the number of permutations on $[n]$ that avoid substrings\linebreak $12, 23,\ldots , (n-1)n, n1$.
  \item[$Der_n$] := the $n$th derangement number, \textit{ie}.
  \begin{equation}\label{eq1}
    Der_{n}=n!\sum_{k=0}^{n}\frac{(-1)^{n}}{k!}.
  \end{equation}
\end{description}

In [2] we discussed the existing result
\begin{equation}\label{eq2}
  d_n=\sum_{j=0}^{n-1}(-1)^j\binom{n-1}{j}(n-j)!.
\end{equation}
We also proved (Equation 2.1)
\begin{equation}\label{eq3}
  D_n=n!\sum_{k=0}^{n-1}\frac{(-1)^k}{k!},
\end{equation}
which is equivalent to
\begin{equation}\label{eq4}
  D_n=\sum_{j=0}^{n-1}(-1)^j\binom{n}{j}(n-j)!.
\end{equation}
Finally in Proposition 2.4, we proved that $D_n = Der_n + (-1)^{n-1}$, $n \geq 1$, which we called the \textquotedblleft alternating derangement sequence\textquotedblright\ since these numbers alternate plus or minus one from the derangement sequence itself. This is sequence A000240 in OEIS [3].

Now we extend the results to forbidden substrings that are not one space apart but $k$ spacings apart (what we call \textquotedblleft $k$-shifts\textquotedblright\ in the following section).

We define a minimal forbidden substring as two consecutive elements $jk$. We assign to this minimal substring a length equal to one. Hence the length of forbidden substrings will be one less than the number of elements.
\section{Main Lemmas and Propositions}
\subsection{Results for $\boldsymbol{\{d_n\}}$ and  $\boldsymbol{\{d_n^k \}}$}

For the sake of compactness, we define $\{d_n\}$ as the set of permutations on $[n]$ that avoid substrings $12, 23,\ldots ,(n-1)n$, with $d_n$ being the number of such permutations.

We generalize to $\{d_n^k \}$, $k<n$, defined as the set of permutations on $[n]$ that for fixed $k$, avoid substrings $j(j+k)$, $1\leq j\leq n-k$. We will refer to these substrings that are $k$ spacings apart as \textquotedblleft$k$-shifts\textquotedblright, or \textquotedblleft$k$-successions\textquotedblright. We let $d_n^k$ be the number of such permutations (the reason for power notation will become apparent in the next section).

The forbidden substrings in these permutations can be pictured as a diagonal running $k$ places to the right of the main diagonal of an $n\times n$ chessboard (hence the term \textquotedblleft$k$-shifts\textquotedblright) as can be seen in Figure \ref{Fig1} below.

\begin{figure}[h!]
  \centering
    \begin{tabular}{l|c|c|c|c|c|c|}
      \multicolumn{1}{c}{} & \multicolumn{1}{c}{1} & \multicolumn{1}{c}{2} & \multicolumn{1}{c}{3} & \multicolumn{1}{c}{4} & \multicolumn{1}{c}{5} & \multicolumn{1}{c}{6} \\
      \cline{2-7}
      1 &  &  & $\times$ &  &  &  \\
      \cline{2-7}
      2 &  &  &  & $\times$ &  &  \\
      \cline{2-7}
      3 &  &  &  &  & $\times$ &  \\
      \cline{2-7}
      4 &  &  &  &  &  & $\times$ \\
      \cline{2-7}
      5 &  &  &  &  &  &  \\
      \cline{2-7}
      6 &  &  &  &  &  &  \\
      \cline{2-7}
    \end{tabular}
  \caption{Forbidden positions in $\{d^2_6\}$.}\label{Fig1}
\end{figure}
\newpage
Figure \ref{Fig1} shows the forbidden positions on a $6\times 6$ chessboard that correspond to  forbidden substrings of permutations in $\{d^2_6\}$. The forbidden substrings are $\{13, 24, 35, 46\}$. Note that there are $n-k$ forbidden substrings in $\{d_n^k\}$.

The permutations that avoid these substrings are not too difficult to handle, and in fact we can count them for any $k$, as we show in the following proposition.

\begin{prop}\label{propo21}
For any fixed $k$, $0 < k < n$, if $d_n^k$ denotes the number of permutations that avoid substrings $j(j+k)$, $1 \leq j \leq n-k$, then
\begin{equation}\label{eq21}
  d_n^k=\sum_{j=0}^{n-k}(-1)^j\binom{n-k}{j}(n-j)!.
\end{equation}
\end{prop}

\begin{proof}
For any fixed $n$ and $k$, there are a total of
\[\binom{n-k}{j}(n-j)!\]
forbidden substrings of length $j$ since the combinatorial term counts the number of ways to get such substrings while the term $(n-j)!$ counts the permutations of the substrings and the remaining elements. Using inclusion-exclusion we get the result.
\end{proof}

We note that the case $k=1$ is just the result we had for $d_n$ in Equation \ref{eq2}.

For example, consider permutations in $\{d_5^3\}$. The $n-k$ forbidden substrings are $\{14, 25\}$. For $j=0$ we get the $5!$ permutations in $S_5$.  For $j = 1$ there are $\binom{2}{1}$ ways to choose one forbidden substring, and we can permute it with the remaining elements in $4!$ ways (for example, select the forbidden substring 14 and permute the 4 blocks 14, 2, 3, 5). For $j = 2$ there are $\binom{2}{2}$ ways to choose two forbidden substrings and we can permute them with the remaining elements in $3!$ ways (that is, permute the 3 blocks 14, 25, 3). Hence $d_5^3= 78$, so there are 78 permutations in $S_5$ that avoid substrings $\{14, 25\}$.

\begin{cor}\label{coro22}
The following relation holds for $d_n^k$\text{\normalfont :}
\begin{equation}\label{eq22}
  d_{n}^{k+1}=d_{n}^{k}+d_{n-1}^{k}.
\end{equation}
\end{cor}

\begin{proof}
By Equation \ref{eq21} and elementary methods.
\end{proof}

Now we define $d_n^0:= Der_n$, which makes sense since in a chessboard of forbidden positions, a derangement is represented by an $X$ in the position $(j, j)$, \textit{ie}. a 0-shift.

Note that Equation \ref{eq22} generalizes the relation in Lemma 2.3 in [2], and we have the following equations starting at $n = 1$:
\begin{align*}
  d_n =&d_n^1 = Der_n+ Der_{n-1}\\
   & d_{n}^{2}= d_{n}+d_{n-1}\\
   & d_{n}^{3}= d_{n}^{2}+d_{n-1}^{2}\ \ \cdots
\end{align*}
Using the inital condition condition $d_2^1 = Der_2$, Equation \ref{eq22} defines a binomial- type relation, which, upon iteration, gives us the triangle in Table \ref{tab1} in the Appendix. Note in particular that for $k = n-1$, $d_n^k= n! - k!$.\footnote[3]{This triangle follows the same recurrence as the so-called Euler's Difference Table, which originally had no combinatorial interpretation. Euler's Table also has a $k!$ term in each column, terms that don't apply in our context of $k$-shifts.}

Note from the triangle that we may get $d_n^k$ starting only from derangement numbers. For example, to get $d_8^5$, \textit{ie}. the number of permutations of length 8 with forbidden substrings $\{16, 27, 38\}$, we can start from the upper-left corner of the table and by successive addition along the triangle we can reach cell $d_8^5= 27,240$ (or we can obviously use Equation \ref{eq21}). Hence there are 27,240 permutations in $S_8$ that avoid substrings $\{16, 27, 38\}$.

For further references, the sequence $\langle d_n^k\rangle$ is available in OEIS. For example, for $k = 3$, the sequence is now A277609 [4].

\subsection{Results for $\boldsymbol{\{D_n\}}$ and $\boldsymbol{\{D_n^k\}}$}

Now we define $\{D_n\}$ as the set of permutations on $[n]$ that avoid substrings $12, 23,\ldots , (n-1)n, n1$, with $D_n$ being the number of such permutations.

We generalize to define permutations without $k$-shifts (or $k$-successions) $(\text{mod }n)$, $\{D^k_n\}$, as the set of permutations on [$n$] that for fixed $k$, $k < n$, avoid substrings $j(j+k)$ for $1 \leq j \leq n-k$, and avoid substrings $j(j+k)$ $(\text{mod }n)$ for \mbox{$n-k < j \leq n$}. Note that we can summarize in the single definition \textquotedblleft avoid substrings\linebreak $j(j+k)$ $(\text{mod }n)$ for all $j$, $1 \leq j \leq n$\textquotedblright\ if we agree to write $n$ instead of 0 when doing addition $(\text{mod }n)$.

We let $D_n^k$ be the number of such permutations. These forbidden substrings are easily seen along an $n\times n$ chessboard, where for $j > n-k$, the forbidden\linebreak\newpage positions start again from the first column along a diagonal $(n-k)$ places below the main diagonal as in Figure \ref{Fig2} below.

\begin{figure}[h!]
  \centering
    \begin{tabular}{l|c|c|c|c|c|c|}
      \multicolumn{1}{c}{} & \multicolumn{1}{c}{1} & \multicolumn{1}{c}{2} & \multicolumn{1}{c}{3} & \multicolumn{1}{c}{4} & \multicolumn{1}{c}{5} & \multicolumn{1}{c}{6} \\
      \cline{2-7}
      1 &  &  & $\times$ &  &  &  \\
      \cline{2-7}
      2 &  &  &  & $\times$ &  &  \\
      \cline{2-7}
      3 &  &  &  &  & $\times$ &  \\
      \cline{2-7}
      4 &  &  &  &  &  & $\times$ \\
      \cline{2-7}
      5 & $\times$ &  &  &  &  &  \\
      \cline{2-7}
      6 &  & $\times$ &  &  &  &  \\
      \cline{2-7}
    \end{tabular}
  \caption{Forbidden positions in $\{D^2_6\}$.}\label{Fig2}
\end{figure}

Figure \ref{Fig2} shows the forbidden positions on a $6\times 6$ chessboard that correspond to  forbidden substrings of permutations in $\{D^2_6\}$.  These forbidden substrings are $\{13, 24, 35, 46; 51, 62\}$. The forbidden substrings below the diagonal are separated by a semicolon; these are the forbidden substrings $j(j+k)\ (\text{mod }n)$ for $n-k<j\leq n$.  Note that while there are only $n-k$ forbidden substrings in $\{d_n^k\}$, there are $n$ forbidden substrings in $\{D_n^k\}$.

It turns out that the numbers $D_n^k$ are more difficult to get. They depend on whether $n$ is prime, and more generally, on whether $n$ and $k$ are relatively prime.
\begin{prop}\label{propo23}
In the set of permutations $\{D_n^k \}$ with $0<k<n$, $k$ relative prime to $n$, $n \geq 2$, we can form a forbidden substring of length \mbox{$j = n-1$}.
\end{prop}

\begin{proof}
Start with forbidden substrings $12, 23,\ldots , (n-1)n$ in $\{D_n\}$ and form the permutation $(12\ldots n)$ in cycle notation. Since $k$-powers of the permutation produce forbidden $k$-shifts (or $k$-successions), we see that the longest cycle of\linebreak forbidden substrings will have length $n/(n,k)$, where $(n,k)$ stands for the greatest common divisor. Hence the longest cycle length of forbidden substrings will be achieved for $(n,k) = 1$, and in this case we will have a cycle of length $n$, which
represents a forbidden substring of length $n-1$.
\end{proof}

Note that the proof of the proposition justifies the power notation in $\{D_n^k\}$ (and in $k$-shifts in general). Note also that for $n \geq 2$, Proposition \ref{propo23} implies that we will always have the longest forbidden substrings of length $n-1$ in $\{D_n\}$, (\textit{ie.}$\ k = 1$), as well as in some other $\{D^k_n\}$ whenever $(n,k) =1$. Moreover, Proposition \ref{propo25} will show that the number of permutations in these sets are equal.

As an example, Figure \ref{Fig3} shows the maximum cycle lengths achieved for the case $n=6$ and all its possible $k$-shifts (or $k$-successions), $k = 1,\ldots , 5$. In this case there exist  maximum length substrings in $\{ D_6\}$ and $\{D^5_6\}$.

\begin{figure}[h!]
  \centering
  \begin{tabular}{|c|c|c|c|c|}
    \hline
    $k$ & Forbidden Substrings & Permutation & $k$th-power & \parbox{1.8cm}{\centering max cycle length} \\
    \hline
    $k=1$ & 12, 23, 34, 45, 56; 61 & (123456) & (123456) & 6 \\
    \hline
    $k=2$ & 13, 24, 35, 46; 51, 62 & (123456) & (135)(246) & 3 \\
    \hline
    $k=3$ & 14, 25, 36; 41, 52, 63 & (123456) & (14)(25)(36) & 2 \\
    \hline
    $k=4$ & 15, 26; 31, 42, 53, 64 & (123456) & (153)(264) & 3 \\
    \hline
    $k=5$ & 16; 21, 32, 43, 54, 65 & (123456) & (165432) & 6 \\
    \hline
  \end{tabular}
  \caption{$k$-shifts and maximum cycle lengths for $n=6$.}\label{Fig3}
\end{figure}

Note that Proposition \ref{propo23} is not true if $k$ is not relative prime to $n$, for example in $\{D^2_6\}$. In this case, the forbidden substrings are $\{13, 24, 35, 46; 51, 62\}$, and the maximum cycle length is 3, which means we can have forbidden substrings of length at most 2 (such as 135 or 246). (Recall that our minimal forbidden substring consists of two elements $jk$, with length equal to one).

\begin{cor}\label{coro24}
In the set of permutations $\{D_n^k \}$ with $k$ relative prime to $n$, $n \geq 2$, there exist forbidden substrings of any length $j$, $j = 1, 2,\ldots , n-1$.
\end{cor}

\begin{proof}
By the previous proposition, for $(n,k) = 1$ we can get the longest forbidden substring of length $j = n-1$. Note that it can be considered either a single substring of length $n-1$ or $n-1$ substrings of length 1. Hence once this substring is obtained, we can split it to get forbidden substrings of any length $j$, $j = 1,\ldots , n-1$.
\end{proof}

\begin{prop}\label{propo25}
The number of permutations in $\{D_n^k \}$ with $0<k<n$,\linebreak $k$ relative prime to $n$, $n \geq 2$, is the same as the number of permutations in $\{D_n\}$.
\end{prop}

\begin{proof}
By the previous Corollary and Proposition, since for $k$'s such that $(n,k) = 1$, we can have forbidden substrings of any length, $j = 1,\ldots , n-1$. It is easy to count that there are exactly $\binom{n}{j}$ ways to get $j$ forbidden substrings (either disjoint or overlapping), and $(n-j)!$ permutations of these substrings and the remaining elements. Then by inclusion-exclusion we have that
\begin{equation}\label{eq23}
    D_n^k=\sum_{j=0}^{n-1}(-1)^j \binom{n}{j}(n-j)!.
\end{equation}
But this is the same as Equation \ref{eq4}, which counts the number of permutations in $\{D_n\}$.
\end{proof}

As an example of the previous proposition, consider again $n = 8$, $k = 5$. In this case, forbidden substrings in $\{D_8^5 \}$ are $\{16, 27, 38; 41, 52, 63, 74, 85\}$. It is easy to count, for instance, that there are $\binom{8}{4}$ forbidden substrings of length $j = 4$ and $(8-4)!$ permutations of these substrings and the remaining elements. For example, a disjoint substring of length 4 (alternatively, four substrings of length 1) is given by 1638, 74 and we count $(8-4)!= 4!$ permutations of the four blocks 1638, 74, 2, 5. Another substring of length 4 is given by 16385 and we count $4!$ permutations of the blocks 16385, 2, 4, 7.  Note that since $(8,5) = 1$, there are $\binom{8}{7}$ substrings of maximum length  $j=n-1=7$ (for example, 16385274), and we can permute these single blocks in $(8-7)!=1$ way since there are no remaining elements to permute them with. Then, by Equation \ref{eq23}, we see that $D_8^5= 14,832$ so there are this number of permutations in $S_8$ that avoid substrings $\{16, 27, 38; 41, 52, 63, 74, 85\}$. Note that $D^5_8<d^5_8$ since there are more forbidden substrings in $\{D^5_8\}$ than in $\{d^5_8\}$.\\[0.1em]
\begin{cor}\label{coro26}
In the set of permutations  $\{D_p^k\}$ with $p$ prime, we can form a substring of length $j = p-1$ for any $k$-shift, $k = 1, 2,\ldots , p-1$.
\end{cor}

\begin{proof}
$(p,k)=1$, $k=1,2,\ldots, p-1$.
\end{proof}

For example, for $p = 5$, we can get longest forbidden substrings of length 4 for all $k$-shifts, $k = 1,2,3,4$ by taking powers of the permutation $(12345)$, as done in Figure \ref{Fig3} above. For $k = 3$, for instance, a longest forbidden substring is given by $(12345)^3 = (14253) \rightarrow 14253$. This substring in $\{D^3_5\}$ corresponds to forbidden positions along a $5\times 5$ chessboard for the $k$-shift $k = 3$, as can be seen in Figure \ref{Fig4}.

\begin{figure}[h!]
  \centering
    \begin{tabular}{l|c|c|c|c|c|}
      \multicolumn{1}{c}{} & \multicolumn{1}{c}{1} & \multicolumn{1}{c}{2} & \multicolumn{1}{c}{3} & \multicolumn{1}{c}{4} & \multicolumn{1}{c}{5} \\
      \cline{2-6}
      1 &  &  &  & $\times$ &  \\
      \cline{2-6}
      2 &  &  &  &  & $\times$ \\
      \cline{2-6}
      3 & $\times$ &  &  &  &  \\
      \cline{2-6}
      4 &  & $\times$ &  &  &  \\
      \cline{2-6}
      5 &  &  & $\times$ &  &  \\
      \cline{2-6}
    \end{tabular}
  \caption{Forbidden positions in $\{D^3_5\}$.}\label{Fig4}
\end{figure}
Proposition \ref{propo25} and Corollary \ref{coro26} in turn imply:

\begin{cor}\label{coro27}
The number of  permutations in $\{D_p^k\}$ for any $k$-shift,\linebreak $k = 1, 2,\ldots ,  p-1$, is the same as the number of permutations in $\{D_p\}$,\linebreak $p$ prime.
\end{cor}

One can see from Table \ref{tab2} in the Appendix, for instance, that for $p = 5$ there are $45$ permutations in $\{D^k_5\}$ for any $k$-shift $k = 1, 2, \ldots, 4$.

The maximum cycle length achieved for a particular $n$ and $k$ is a very important statistic. In fact, for any fixed $n$, $k$-shifts that have the same maximum cycle length will produce the same number of permutations, as can be seen in\linebreak Table \ref{tab2} in the Appendix.\footnote[4]{No similar table appears in other references to our knowledge.} This table shows that there will be the same number of permutations in $\{D^{k_1}_n\}$ and $\{D^{k_2}_n\}$ whenever $(n, k_1) = (n, k_2)$, since then the maximum cycle lengths will be equal.
%
%

\newpage
\section*{References}

{\small

[1] R.A. Brualdi, Introductory Combinatorics (1992), 2nd edition.

[2] E. Navarrete, Forbidden Patterns and the Alternating Derangement Sequence, 2016.
\hspace*{0.5cm}arXiv:1610.01987 [math.CO], 2016.

[3] N. J. A. Sloane, The On-Line Encyclopedia of Integer Sequences, published electronically
\hspace*{0.5cm}at http://oeis.org, sequence A000240.

[4] N. J. A. Sloane, The On-Line Encyclopedia of Integer Sequences, published electronically
\hspace*{0.5cm}at http://oeis.org, sequence A277609.

[5]  R.P. Stanley, Enumerative Combinatorics, Vol. 1 (2011), 2nd edition.
}
\begin{center}
{\Large \textbf{APPENDIX}}
\end{center}
\appendix
\begin{table}[h!]
  \centering
  \begin{tabular}{|c|r|r|r|r|r|r|}
     \hline
      \hspace{0.37cm}$n$\hspace{0.37cm} & \hspace{0.18cm}$Der_n$\hspace{0.18cm} & \hspace{0.38cm}$d_n$\hspace{0.38cm} & \hspace{0.38cm}$d_n^2$\hspace{0.38cm} & \hspace{0.38cm}$d_n^3$\hspace{0.38cm} & \hspace{0.38cm}$d_n^4$\hspace{0.38cm} & \hspace{0.38cm}$d_n^5$\hspace{0.38cm} \\
     \hline
     1& 0& & & & &\\
     2& 1& 1& & & &\\
     3& 2& 3& 4& & &\\
     4& 9& 11& 14& 18& &\\
     5& 44& 53& 64& 78& 96&\\
     6& 265& 309& 362& 426& 504& 600\\
     7& 1,854& 2,119& 2,428& 2,790& 3,216& 3,720\\
     8& 14,833& 16,687& 18,806& 21,234& 24,024& 27,240\\
     \hline
   \end{tabular}
  \caption{Some values for $d_n^k$.}\label{tab1}
\end{table}

\begin{table}[h!]
  \centering
  \begin{tabular}{|c|r|r|r|r|r|r|}
     \cline{2-7}
      \multicolumn{1}{c|}{}& $k=1$\hspace{0.1cm} & $k=2$\hspace{0.1cm} & $k=3$\hspace{0.1cm} & $k=4$\hspace{0.1cm} & $k=5$\hspace{0.1cm} & $k=6$\hspace{0.1cm} \\
     \hline
$n = 2$& 0& & & & &\\
$n = 3$& 3& 3& & & &\\
$n = 4$& 8& 8& 8& & &\\
$n = 5$& 45& 45& 45& 45& &\\
$n = 6$& 264& 270& 240& 270& 264&\\
$n = 7$& 1,855& 1,855& 1,855& 1,855& 1,855& 1,855\\
$n = 8$& 14,832& 14,816& 14,832& 13,824& 14,832& 14,816\\
$n = 9$& 133,497& 133,497& 134,298& 133,497& 133,497& 134,298\\
     \hline
   \end{tabular}
  \caption{Some values for $D_n^k$.}\label{tab2}
\end{table}


\end{document}